\documentclass[review]{elsarticle}

\usepackage{srcltx}
\usepackage{eurosym}
\usepackage{mathtools}
\usepackage{amsmath}
\usepackage{amsfonts}
\usepackage{amssymb}
\usepackage{amsthm}
\usepackage{graphicx}
\usepackage{mathrsfs}
\usepackage{xcolor}
\usepackage{exscale}
\usepackage{latexsym}
\usepackage[colorlinks,plainpages=true,pdfpagelabels,hypertexnames=true,colorlinks=true,pdfstartview=FitV,linkcolor=blue,citecolor=red,urlcolor=black]{hyperref}
\PassOptionsToPackage{unicode}{hyperref}
\PassOptionsToPackage{naturalnames}{hyperref}
\usepackage{enumerate}
\usepackage[shortlabels]{enumitem}
\usepackage{bookmark}
\usepackage{wasysym}
\usepackage{esint}
\usepackage[ddmmyyyy]{datetime}
\usepackage[margin=2cm]{geometry}
\makeatletter
\g@addto@macro\normalsize{%
  \setlength\abovedisplayskip{4pt}
  \setlength\belowdisplayskip{4pt}
  \setlength\abovedisplayshortskip{4pt}
  \setlength\belowdisplayshortskip{4pt}
}
\makeatother
\numberwithin{equation}{section}
\everymath{\displaystyle}
\usepackage[capitalize,nameinlink]{cleveref}
\crefname{section}{Section}{Sections}
\crefname{subsection}{Subsection}{Subsections}
\crefname{condition}{Condition}{Conditions}
\crefname{hypothesis}{Hypothesis}{Conditions}
\crefname{assumption}{Assumption}{Assumptions}
\crefname{lemma}{Lemma}{Lemmas}

\crefformat{equation}{\textup{#2(#1)#3}}
\crefrangeformat{equation}{\textup{#3(#1)#4--#5(#2)#6}}
\crefmultiformat{equation}{\textup{#2(#1)#3}}{ and \textup{#2(#1)#3}}
{, \textup{#2(#1)#3}}{, and \textup{#2(#1)#3}}
\crefrangemultiformat{equation}{\textup{#3(#1)#4--#5(#2)#6}}%
{ and \textup{#3(#1)#4--#5(#2)#6}}{, \textup{#3(#1)#4--#5(#2)#6}}%
{, and \textup{#3(#1)#4--#5(#2)#6}}

\Crefformat{equation}{#2Equation~\textup{(#1)}#3}
\Crefrangeformat{equation}{Equations~\textup{#3(#1)#4--#5(#2)#6}}
\Crefmultiformat{equation}{Equations~\textup{#2(#1)#3}}{ and \textup{#2(#1)#3}}
{, \textup{#2(#1)#3}}{, and \textup{#2(#1)#3}}
\Crefrangemultiformat{equation}{Equations~\textup{#3(#1)#4--#5(#2)#6}}%
{ and \textup{#3(#1)#4--#5(#2)#6}}{, \textup{#3(#1)#4--#5(#2)#6}}%
{, and \textup{#3(#1)#4--#5(#2)#6}}

\crefdefaultlabelformat{#2\textup{#1}#3}
\newtheorem{theorem}{Theorem}[section]
\newtheorem{lemma}[theorem]{Lemma}

\newtheorem{definition}[theorem]{Definition}
\newtheorem{remark}[theorem]{Remark}        

\numberwithin{equation}{section}
\def\aa{\mathcal{A}}

\newcommand{\lamot}{\La_0,\La_1}

\newcommand{\pa}{\partial}

\newcommand{\tQ}{\tilde{Q}}

\makeatletter
\newcommand{\vo}{\vec{o}\@ifnextchar{^}{\,}{}}
\makeatother
\def\Yint#1{\mathchoice
    {\YYint\displaystyle\textstyle{#1}}%
    {\YYint\textstyle\scriptstyle{#1}}%
    {\YYint\scriptstyle\scriptscriptstyle{#1}}%
    {\YYint\scriptscriptstyle\scriptscriptstyle{#1}}%
      \!\iint}
\def\YYint#1#2#3{{\setbox0=\hbox{$#1{#2#3}{\iint}$}
    \vcenter{\hbox{$#2#3$}}\kern-.50\wd0}}
\def\longdash{-\mkern-9.5mu-} 
\def\tiltlongdash{\rotatebox[origin=c]{18}{$\longdash$}}
\def\fiint{\Yint\tiltlongdash}
\def\Xint#1{\mathchoice
    {\XXint\displaystyle\textstyle{#1}}%
    {\XXint\textstyle\scriptstyle{#1}}%
    {\XXint\scriptstyle\scriptscriptstyle{#1}}%
    {\XXint\scriptscriptstyle\scriptscriptstyle{#1}}%
      \!\int}
\def\XXint#1#2#3{{\setbox0=\hbox{$#1{#2#3}{\int}$}
    \vcenter{\hbox{$#2#3$}}\kern-.50\wd0}}
\def\hlongdash{-\mkern-13.5mu-}
\def\tilthlongdash{\rotatebox[origin=c]{18}{$\hlongdash$}}
\def\hint{\Xint\tilthlongdash}
\makeatletter
\def\namedlabel#1#2{\begingroup
   \def\@currentlabel{#2}%
   \label{#1}\endgroup
}
\makeatother
\makeatletter
\newcommand{\rmh}[1]{\mathpalette{\raisem@th{#1}}}
\newcommand{\raisem@th}[3]{\hspace*{-1pt}\raisebox{#1}{$#2#3$}}
\makeatother

\newcommand{\lsb}[2]{#1_{\rmh{-3pt}{#2}}}

\newcommand{\redref}[2]{\texorpdfstring{\protect\hyperlink{#1}{\textcolor{black}{(}\textcolor{red}{#2}\textcolor{black}{)}}}{}}
\newcommand{\redlabel}[2]{\hypertarget{#1}{\textcolor{black}{(}\textcolor{red}{#2}\textcolor{black}{)}}}

\newcommand{\descref}[2]{\hyperref[#1]{\textnormal{\textcolor{black}{(}\textcolor{blue}{\bf #2}\textcolor{black}{)}}}}

\newcommand{\dref}[2]{\hyperref[#1]{\textcolor{black}{(}\textcolor{blue}{\bf #2}\textcolor{black}{)}}}

\newcommand\RR{\mathbb{R}}

\newcommand\NN{\mathbb{N}}
\newcommand{\al}{\alpha}

\newcommand{\de}{\delta}
\newcommand{\ve}{\varepsilon}

\newcommand{\tht}{\theta}
\newcommand{\ep}{\epsilon}

\newcommand{\la}{\lambda}

\newcommand{\Om}{\Omega}

\newcommand{\La}{\Lambda}

\DeclareMathOperator{\dv}{div}

\DeclareMathOperator{\loc}{loc}

\newcommand{\iprod}[2]{\langle #1,#2\rangle}

\newcommand{\abs}[1]{\left| #1\right|}

\newcommand{\lbr}[1][(]{\left#1}
\newcommand{\rbr}[1][)]{\right#1}

\newcommand{\txt}[1]{\qquad \text{#1} \qquad}
\newcommand{\rhot}{\tilde{\rho}}
\newcommand{\thtt}{\tilde{\tht}}

\newcommand{\ukm}[1]{(u-#1)_+}
\newcounter{whitney}
\refstepcounter{whitney}

\newcounter{ineqcounter}
\refstepcounter{ineqcounter}
\makeatletter
\def\ps@pprintTitle{%
\let\@oddhead\@empty
\let\@evenhead\@empty
\def\@oddfoot{}%
\let\@evenfoot\@oddfoot}
\makeatother
\usepackage[doublespacing]{setspace}
\usepackage[titletoc,toc,page]{appendix}

\begin{document}

\begin{frontmatter}

\title{Uniform boundedness for weak solutions of quasilinear parabolic equations}

\author[myaddress]{Karthik Adimurthi\corref{mycorrespondingauthor}\tnoteref{thanksfirstauthor}}
\cortext[mycorrespondingauthor]{Corresponding author}
\ead{karthikaditi@gmail.com and kadimurthi@snu.ac.kr}
\tnotetext[thanksfirstauthor]{Supported by the National Research Foundation of Korea grant NRF-2015R1A4A1041675.}

\author[myaddresstwo]{Sukjung Hwang\tnoteref{thankssecondauthor}}
\ead{sukjung\_hwang@yonsei.ac.kr and sukjungh@gmail.com}
\tnotetext[thankssecondauthor]{Supported by Basic Science Research Program
        through the National Research Foundation of Korea(NRF) funded
        by the Ministry of Education, Science and Technology(2017R1D1A1B03035152).}

\address[myaddress]{Department of Mathematical Sciences, Seoul National University, Seoul 08826, Republic of Korea.}
\address[myaddresstwo]{Department of Mathematics, Yonsei University, Seoul 03722, Republic of Korea.}

\begin{abstract}
In this paper, we study the boundedness of weak solutions to quasilinear parabolic equations of the form
\[
    u_t - \dv \aa(x,t,\nabla u) = 0,
\]
where the nonlinearity $\aa(x,t,\nabla u)$ is modelled after the well studied $p$-Laplace operator. The question of boundedness has received lot of attention over the past several decades with the existing literature showing that weak solutions in either $\frac{2N}{N+2}<p<2$, $p=2$ or $2<p$ are bounded. The proof is essentially split into three cases mainly because the estimates that have been obtained in the past always included an exponent of the form $\frac{1}{p-2}$ or $\frac{1}{2-p}$ which blows up as $p \rightarrow 2$.  In this note, we prove the boundedness of weak solutions in the full range $\frac{2N}{N+2} < p < \infty$ without having to consider the singular and degenerate cases separately. Subsequently, in a slightly smaller regime of $\frac{2N}{N+1} < p < \infty$, we also prove an improved boundedness estimate.
\end{abstract}


\end{frontmatter}


\section{Introduction}
In this paper, we study weak solutions $u \in L^p(-T,T; W^{1,p}_{\loc}(\Om))$ of
\begin{equation}
    \label{main}
    u_t - \dv \aa(x,t,\nabla u) = 0 \txt{in} \mathcal{D}'(\Om_T),
\end{equation}
where $\aa(x,t,\nabla u)$ is modelled after the well known $p$-Laplace operator. More specifically, we assume that the nonlinear structure satisfies the following growth and coercivity conditions for some $p>1$ and positive constants $0<\La_0\leq \La_1<\infty$:
\begin{equation*}
\iprod{\aa(x,t,\nabla u)}{\nabla u} \geq \La_0 |\nabla u|^p \txt{and} |\aa(x,t,\nabla u)| \leq \La_1 |\nabla u|^{p-1}.
\end{equation*}

Over the past several decades, there has been much progress made regarding the regularity of weak solutions, but we shall refrain from giving a comprehensive history regarding the development of boundedness estimates for solutions of \cref{main} and refer to \cite[Section 18 of Chapter V]{DiBenedetto} and references therein for more about the history of the problem.

The boundedness in existing literature takes the following form:
\begin{theorem}[Degenerate case]
\label{deg_bound}
    Here we have $p >2$ and let $\sigma \in (0,1)$ be given and $\theta, \rho \in (0,\infty)$ be any two positive constants. Then for any $\ve \in (0,2]$, any non-negative weak solution $u \in L^p(-T,T;W^{1,p}_{\loc}(\Om))$ of \cref{main} satisfies
    \[
        \sup_{Q_{\sigma\rho,\sigma\tht}} |u| \leq C(N,p,\La_0,\La_1,\ve) \frac{\lbr \frac{\tht}{\rho^p} \rbr^{\frac{1}{\ve}}}{(1-\sigma)^{\frac{N+p}{\ve}}} \lbr \fiint_{Q_{\rho,\tht}} |u|^{p-2+\ve} \ dz \rbr^{\frac{1}{\ve}} \bigwedge \lbr \frac{\rho^p}{\tht}\rbr^{\frac{1}{p-2}}.
    \]
\end{theorem}

\begin{theorem}[Singular case]
\label{sing_bound}
Here, we take $1<p<2$ and set $\la_r := N(p-2) + rp$ for some $r \geq 1$ such that $\la_r >0$. Then any non-negative  weak solution $u \in L^p(-T,T;W^{1,p}_{\loc}(\Om))$ of \cref{main} satisfies
    \[
        \sup_{Q_{\sigma\rho,\sigma\tht}} |u| \leq C(N,p,\La_0,\La_1,\ve) \frac{\lbr \frac{\rho^p}{\tht} \rbr^{\frac{N}{\la_r}}}{(1-\sigma)^{\frac{p(N+p)}{\la_r}}} \lbr \fiint_{Q_{\rho,\tht}} |u|^{r} \ dz \rbr^{\frac{p}{\la_r}} \bigwedge \lbr \frac{\tht}{\rho^p}\rbr^{\frac{1}{2-p}},
    \]
    where $\sigma \in (0,1)$ and $\rho,\tht \in (0,\infty)$ are  positive constants.
\end{theorem}

From \cref{deg_bound} and \cref{sing_bound}, we see that the estimates are  unstable as $p \rightarrow 2$ and hence the proofs are different in the singular and degenerate regimes. In this paper we overcome this trichotomy by proving two boundedness results, both of which are stable with respect to $p \rightarrow 2$.  The first boundedness result is the following:
\begin{theorem}
\label{main_thm1}
Let $\frac{2N}{N+2}<p<\infty$ and $\ve_0:= \frac{4}{N+2}$. Let $\sigma \in (0,1)$ and $\rho,\tht \in (0,\infty)$ be given, then any non-negative weak solution $u \in L^p(-T,T;W^{1,p}_{\loc}(\Om))$ of \cref{main} satisfies
    \[
\sup_{Q_{\sigma \rho,\sigma \tht}} u  \leq   \mathbf{C} (N,p,\lamot) \frac{\aa^{\frac{p(N+2)}{2(p(N+2)-2N)}}}{(1-\sigma)^{\frac{p(N+p)(N+2)}{2\left(p(N+2)-2N\right)}}}  \lbr \fiint_{Q_{\rho,\tht}}|u|^{p+\ve_0} \,dz\rbr^{\frac{p(N+2)}{2(p(N+2)-2N)}} \bigwedge 1.
\]
Here we have set 
\begin{equation}
\label{def_aa_func}
    \aa := \lbr \frac{\tht}{\rho^p} \rbr + \lbr \frac{\rho^p}{\tht} \rbr^{\frac{N}{p}}.
\end{equation}
\end{theorem}
 
 From \cref{par_sob_emb}, we see that in the range $\frac{2N}{N+2}<p<\infty$, any weak solution $u \in L^p(-T,T;W^{1,p}_{\loc}(\Om))$ is actually in $L^{p\frac{N+2}{N}}(\Om_T)$ and the choice of $\ve_0 = \frac{4}{N+2}$ is taken such that 
 \[
     p+ \frac{4}{N+2} < p \frac{N+2}{N} \txt{and} 2 < p+ \frac{4}{N+2}.
 \]
 
Subsequently, using a second iteration, we are able to prove the following ameliorated estimate:
\begin{theorem}
    \label{main_thm2}
In the range $\frac{2N}{N+1}<p<\infty$,  let $\sigma \in (0,1)$ and $\rho,\tht \in (0,\infty)$ be given, then any non-negative weak solution $u \in L^p(-T,T;W^{1,p}_{\loc}(\Om))$ of \cref{main} satisfies
    \[
\sup_{Q_{\sigma \rho, \sigma \tht}} u \leq  \mathbf{C} (N,p,\lamot)  \frac{ \aa^{\frac{p(N+1)}{2N(p-1)}}}{\sigma^{\frac{p(N+1)^2}{2N(p-1)}}(1-\sigma)^{\frac{p(N+p)(N+1)}{2N(p-1)} }}\lbr \fiint_{Q_{\rho,\tht}}|u|^{p} \,dz\rbr^{\frac{p(N+1)}{2N(p-1)}}   \bigwedge 1,
\]
where $\aa$ is as defined in \cref{def_aa_func}.
\end{theorem}


The main idea of our proof lies in the parabolic Sobolev embedding, which gives an improved integrability for the weak solution for free. The older proofs of \cref{deg_bound,sing_bound} only requires the finiteness of $u$ in  $L^p(-T,T;L^p_{\loc}(\Om))$, whereas we make use of all the available information, i.e., we make use of  the finiteness of $u$ in the function space $L^p(-T,T;W^{1,p}_{\loc}(\Om))$. This additional information enables us to estimate the two contrasting terms, one with $p$ growth and the other with $2$ growth uniformly (see \cref{2.11} and \cref{2.12} for how the estimates work).  The exponent $\ve_0$ can be viewed as the positive gap between $2$ and the Sobolev exponent $q = p \frac{N+2}{N}>2$ and it is this gap that plays a crucial role in our proof.

\section{Preliminaries}

In this section, we shall collect all the preliminary material needed in subsequent sections. For any $1<p<\infty$ and any $m > 1$, we define the following Banach spaces:
\begin{equation*}
\begin{array}{c}
V^{m,p}(\Om_T) := L^{\infty}(-T,T;L^m(\Om)) \cap L^p(-T,T;W^{1,p}(\Om)),\\
V^{m,p}_0(\Om_T) := L^{\infty}(-T,T;L^m(\Om)) \cap L^p(-T,T;W^{1,p}_0(\Om)).
\end{array}\end{equation*}

We have the following parabolic Sobolev embedding theorem from \cite[Proposition 3.1 from Section I]{DiBenedetto}.
\begin{lemma}
\label{par_sob_emb}
For any $u \in V^{2,p}_0(\Om_T)$, there exists a constant $C = C(N,p)$ such that
\[
\iint_{\Om_T} |u(x,t)|^q \,dx\,dt \apprle C \lbr \sup_{0<t<T} \int_{\Om} |u(x,t)|^2 \,dx \rbr^{\frac{p}{N}} \lbr \iint_{\Om_T} |\nabla u(x,t)|^p \,dx\,dt \rbr,
\]
where $q:= p \frac{N+2}{N}$.
\end{lemma}

Let us first define Steklov average as follows: let $h \in (0,2T)$ be any positive number, then we define
\begin{equation*}
  u_{h}(\cdot,t) := \left\{ \begin{array}{ll}
                              \hint_t^{t+h} u(\cdot, \tau) \ d\tau \quad & t\in (-T,T-h), \\
                              0 & \text{else}.
                             \end{array}\right.
 \end{equation*}
We shall now define the notion of weak solutions to \cref{main}.
\begin{definition}[Weak solution]
\label{weak_sol}
    We say that $u \in C^0(-T,T;L^2_{\loc}(\Om)) \cap L^p(-T,T;W^{1,p}_{\loc}(\Om))$ is a weak solution of \cref{main} if,  for any $\phi \in C_c^{\infty}(\Om)$ and any $t \in (-T,T)$,  the following holds:
\begin{equation*}
  \int_{\Om \times \{t\}} \left\{ \frac{d [u]_{h}}{dt} \phi + \iprod{[\aa(x,t,\nabla u)]_{h}}{\nabla \phi} \right\} \,dx = 0 \txt{for any}0 < t < T-h.
\end{equation*}
\end{definition}

\begin{remark}
Since the boundedness result is local in nature, without loss of generality,  we shall assume that all the cylinders are centered at the point $(0,0)$ after suitable translation of the problem. In what follows, we shall use the following notation:
\begin{gather*}
Q_{\rho,\tht}:= B_{\rho} \times [-\tht, \tht].
\end{gather*}
Furthermore, we shall denote $\pa_p Q_{\rho,\tht}$ to denote the parabolic boundary of the cylinder and denote $z= (x,t)$ to be a point in $\RR^{n+1}$.
\end{remark}

Let us first recall the standard energy estimate (see for example \cite[Proposition 3.1 from Section II]{DiBenedetto} for the proof):
\begin{lemma}
\label{energy_est}
Let $u \in C^0(-T,T;L^2_{\loc}(\Om)) \cap L^p(-T,T;W^{1,p}_{\loc}(\Om))$ be a nonnegative, weak solution of \cref{main} in the sense of  \cref{weak_sol}, then for any $k \in \RR$, there exists a constant $C = C(N,p,\lamot)$ such that
\begin{equation*}
\begin{array}{rcl}
\sup_{-\tht<t<\tht} \int\limits_{B_{\rho}} (u-k)_+^2 \zeta^2 \,dx + \iint\limits_{Q_{\rho,\tht}} |\nabla (u-k)_+|^p \zeta^p \,dz  \apprle \iint\limits_{Q_{\rho,\tht}} | (u-k)_+|^p |\nabla \zeta|^p \,dz + \iint\limits_{Q_{\rho,\tht}} |(u-k)_+|^2 \zeta^{p-1} |\pa_t\zeta| \,dz.
\end{array}
\end{equation*}
Here $\zeta \in C^{\infty}(Q_{\rho,\tht})$ is a cut-off function such that $\zeta = 0$ on $\pa_p Q_{\rho,\tht}$ (the parabolic boundary) for all $t \in (-\tht,\tht)$.
\end{lemma}

We now recall the following well known lemma concerning the geometric convergence of sequence of numbers (see \cite[Lemma 4.1 from Section I]{DiBenedetto} for the details): 
\begin{lemma}\label{geo_con}
Let $\{Y_n\}$, $n=0,1,2,\ldots$, be a sequence of positive number, satisfying the recursive inequalities 
\[ Y_{n+1} \leq C b^n Y_{n}^{1+\alpha}\]
where $C > 1$, $b>1$, and $\alpha > 0$ are given numbers. If 
\[ Y_0 \leq  C^{-\frac{1}{\alpha}}b^{-\frac{1}{\alpha^2}},\]
then $\{Y_n\}$ converges to zero as $n\to \infty$. 
\end{lemma}

\section{Local iterative estimates}
\label{sec3}

Henceforth, we will fix $\sigma \in (0,1)$ and $\rho,\tht \in (0,\infty)$. For $i \in \NN$, we define
\begin{equation*}
\rho_i := \sigma \rho + \frac{(1-\sigma)\rho}{2^i} \txt{and} \tht_i := \sigma \tht + \frac{(1-\sigma)\tht}{2^i}.
\end{equation*}
Corresponding to these radii, we have the following nested sequence of cylinders
\begin{gather*}
Q_i:= Q_{\rho_i, \theta_i} \txt{with} Q_0= Q_{\rho,\tht} \txt{and} Q_{\infty}= Q_{\sigma \rho, \sigma \tht}.
\end{gather*}
We shall define the following radii:
\begin{gather*}
\rhot_i:= \frac{\rho_i + \rho_{i+1}}{2} \txt{and} \thtt_i:= \frac{\tht_i  + \tht_{i+1}}{2}.
\end{gather*}
It is then easy to see that the following holds:
\begin{equation}
\label{tq_def}
Q_{i+1} \subset \tQ_i :=Q_{\rhot_i,\thtt_i} \subset Q_i.
\end{equation}
Subordinate to the cylinders defined in \cref{tq_def}, we  consider the following sequence of cut-off functions $\{ \zeta_i\}$ for $i\in \NN$:
\begin{equation*}
\zeta_i \in C_c^{\infty} (Q_i) \txt{with} \zeta_i = \left\{
\begin{array}{l}
 1 \txt{on} \tQ_i, \\
 0 \txt{on} \pa_p Q_i.
 \end{array}\right.
 \end{equation*}
 Moreover, the cut-off functions $\{\zeta_i\}$ satisfies
 \begin{equation}\label{zeta_bound}
 |\nabla \zeta_{i}| \leq \frac{2^{i+1}}{(1-\sigma)\rho} \txt{and} | \partial_{t} \zeta_i| \leq \frac{2^{i+2}}{(1-\sigma)\tht}.
 \end{equation}

Let us make the following choice of exponents: let $q:= p \frac{N+2}{N}$ be the Sobolev exponent such that $V_0^{2,p} \subset L^q$ with $q>2$. Denote $\ve_0$ to be a positive constant such that
\begin{equation}\label{def_ve_0}
    q > p+ \ve_0 \txt{and} p+\ve_0 > 2.
\end{equation}
In particular, this would require $\max\{2-\ve_0,q-\ve_0\} < p < \infty$ to hold. In \cref{F:imp_bound} and \cref{S:imp_bound}, we shall make more precise choices of $\ve_0$ and the range of $p$ according to the hypothesis of \cref{main_thm1} and \cref{main_thm2}.


Let $k>0$ be a large constant to be eventually chosen and denote for $i = \NN$, 
\begin{gather}
k_i := k - \frac{k}{2^{i}} \txt{ which gives} k_0 = 0 \txt{and} k_{\infty} = k. \label{def_k_i}
\end{gather}
Let us now denote the superlevel sets of $u$ by
\begin{equation}
\label{2.9}
A_{i+1}:= \{ z \in Q_{i+1}: u(z) > k_{i+1}\}.
\end{equation}

We will need the following useful estimate (see \cite[Equation (7.2) of Section V]{DiBenedetto} for the details):
\begin{lemma}
Let $k_i$ be as in \cref{def_k_i}, then for any $s \geq 1$,  there holds
\begin{equation}
\label{2.10}
 \abs{A_{i+1}} \leq \frac{2^{s(i+1)}}{k^s}\iint_{Q_i} \ukm{k_i}^s \,dz,
\end{equation}
where we have used the notation $\ukm{k_i} := (u-k_i) \lsb{\chi}{\{u > k_i\}}$.
\end{lemma}
\begin{proof}
We have the following sequence of estimates
\begin{equation*}
\iint_{Q_i} \ukm{k_i}^s \,dz  \geq  \iint_{Q_i} \ukm{k_i}^s \,\lsb{\chi}{\{u > k_{i+1}\}} \,dz \geq (k_{i+1} - k_i)^s |A_{i+1}| = \frac{k^s}{2^{s(i+1)}}|A_{i+1}|.
\end{equation*}
\end{proof}

Let us apply \cref{energy_est} over cylinders from \cref{tq_def} with $k_i$ defined as in \cref{def_k_i} and estimate each of the terms appearing on the right hand side of \cref{energy_est} as follows:
\begin{description}[leftmargin=*]
\item[First term:] We have the following sequence of estimates:
\begin{equation}\label{2.11}\hspace*{-0.8cm}
\begin{array}{rcl}
\iint_{Q_i} \ukm{k_{i+1}}^p \,dz & \overset{\redlabel{2.11.a}{a}}{\leq} & \lbr \iint_{Q_i} \ukm{k_{i+1}}^{p+ \ve_0} \,dz \rbr^{\frac{p}{p+\ve_0}} \lbr \iint_{Q_i} \lsb{\chi}{\{u>k_{i+1}\}}\,dz \rbr^{\frac{\ve_0}{p+\ve_0}}\\
& \overset{\redlabel{2.11.b}{b}}{=} & \lbr \iint_{Q_i} \ukm{k_{i+1}}^{p+ \ve_0} \,dz \rbr^{\frac{p}{p+\ve_0}} \abs{A_{i+1}}^{\frac{\ve_0}{p+\ve_0}}\\
& \overset{\cref{2.10}}{\leq} & \lbr \iint_{Q_i} \ukm{k_{i+1}}^{p+ \ve_0} \,dz \rbr^{\frac{p}{p+\ve_0}} \lbr \frac{2^{(p+\ve_0)(i+1)}}{k^{p+\ve_0}} \iint_{Q_i} \ukm{k_{i+1}}^{p+ \ve_0} \,dz \rbr^{\frac{\ve_0}{p+\ve_0}}\\
& = & \frac{2^{\ve_0(i+1)}}{k^{\ve_0}} \iint_{Q_i} \ukm{k_{i}}^{p+ \ve_0} \,dz.
\end{array}
\end{equation}
To obtain \redref{2.11.a}{a}, we made us of H\"older's inequality and to obtain \redref{2.11.b}{b}, we used the definition of $A_{i+1}$ from \cref{2.9}.

\item[Second term:] Similarly, we estimate the second term as
\begin{equation}
\label{2.12}\hspace*{-0.8cm}
\begin{array}{rcl}
\iint_{Q_i} \ukm{k_{i+1}}^2 \,dz & \overset{\redlabel{2.12.a}{a}}{\leq} & \lbr \iint_{Q_i} \ukm{k_{i+1}}^{p+ \ve_0} \,dz \rbr^{\frac{2}{p+\ve_0}} \lbr \iint_{Q_i} \lsb{\chi}{\{u>k_{i+1}\}}\ dz \rbr^{\frac{p+\ve_0-2}{p+\ve_0}}\\
& \overset{\redlabel{2.12.b}{b}}{=}& \lbr \iint_{Q_i} \ukm{k_{i+1}}^{p+ \ve_0} \,dz \rbr^{\frac{2}{p+\ve_0}} \abs{A_{i+1}}^{\frac{p+\ve_0-2}{p+\ve_0}}\\
& \overset{\cref{2.10}}{\leq} & \lbr \iint_{Q_i} \ukm{k_{i+1}}^{2+ \ve_0} \,dz \rbr^{\frac{2}{p+\ve_0}} \lbr \frac{2^{(p+\ve_0)(i+1)}}{k^{p+\ve_0}} \iint_{Q_i} \ukm{k_{i+1}}^{p+ \ve_0} \,dz \rbr^{\frac{p+\ve_0-2}{p+\ve_0}}\\
& = & \frac{2^{(p+\ve_0-2)(i+1)}}{k^{p+\ve_0-2}} \iint_{Q_i} \ukm{k_{i}}^{p+ \ve_0} \,dz.
\end{array}
\end{equation}
To obtain \redref{2.12.a}{a}, we note that $p+\ve_0>2$ due to \cref{def_ve_0}  which enables us to apply H\"older's inequality and to obtain \redref{2.12.b}{b}, we used the definition of $A_{i+1}$ from \cref{2.9}.
\end{description}

Substituting \cref{2.11} and \cref{2.12} into the estimate from \cref{energy_est} and removing derivatives of the cut-off function using \cref{zeta_bound}, we get
\begin{equation}
\label{2.14}
\begin{array}{ll}
 & \hspace*{-2cm}\sup_{-\tilde{\tht}_i<t<\tht_i}  \int_{{B}_{\tilde{\rho}_i}} (u-k_{i+1})_+^2  \,dx  +   \iint_{\tilde{Q}_{i}} |\nabla (u-k_{i+1})_+|^p  \,dz \\
 &\qquad \qquad \leq  C  \lbr \frac{2^{(i+2)p}}{(1-\sigma)^p \rho^p} \frac{2^{\ve_0 (i+1)}}{k^{\ve_0}} + \frac{2^{i+2}}{(1-\sigma)\tht}\frac{2^{(p+\ve_0-2)(i+1)}}{k^{p+\ve_0-2}} \rbr \iint_{Q_i} \ukm{k_{i}}^{p+ \ve_0} \,dz \\
 &\qquad \qquad \leq C \frac{2^{(i+2)(p+\ve_0)}}{(1-\sigma)^p }  \lbr   \frac{1}{\rho^pk^{\ve_0}} + \frac{1}{\tht k^{p+\ve_0-2}} \rbr \iint_{Q_i} \ukm{k_{i}}^{p+ \ve_0} \,dz.
\end{array}
\end{equation}

Let us now define
\begin{equation}\label{Y_i}
Y_i:= \fiint_{Q_i} \ukm{k_i}^{p+\ve_0} \,dz,
\end{equation}
then we have the following sequence of estimates
\begin{equation}
\label{2.15}
\begin{array}{rcl}
Y_{i+1} & \overset{\redlabel{2.15.a}{a}}{\leq} & \fiint_{\tilde{Q}_{i}} \ukm{k_{i+1}}^{p+\ve_0} \tilde{\zeta}_i^{p+\ve_0} \,dz \\
& \overset{\redlabel{2.15.b}{b}}{\leq} & C \lbr \fiint_{\tilde{Q}_{i}} \ukm{k_{i+1}}^{q} \tilde{\zeta}_i^{q} \,dz\rbr^{\frac{p+\ve_0}{q}} \lbr[[] \frac{\abs{A_{i+1}}}{|Q_i|}\rbr[]]^{\frac{q-(p+\ve_0)}{q}} \\
& \overset{\redlabel{2.15.c}{c}}{\leq}& C\lbr \fiint_{\tilde{Q}_{i}} \ukm{k_{i+1}}^{q} \tilde{\zeta}_i^{q} \,dz\rbr^{\frac{p+\ve_0}{q}} \lbr[[] \frac{2^{(i+1)(p+\ve_0)}}{k^{p+\ve_0}} \,Y_i\rbr[]]^{\frac{q-(p+\ve_0)}{q}}.
\end{array}
\end{equation}
To obtain \redref{2.15.a}{a}, we used the definition from \cref{Y_i} along with the cut-off function $\tilde{\zeta}_i \in C_c^{\infty}$ such that 
\begin{equation*}
\tilde{\zeta}_i = \left\{
\begin{array}{l}
 1 \txt{on} Q_{i+1}, \\
 0 \txt{on the lateral boundary of} \tQ_i,
 \end{array}\right.
 \end{equation*}
 satisfying $|\nabla \tilde{\zeta}_i| \leq \frac{2^{i+2} }{(1-\sigma)\rho}$. To obtain \redref{2.15.b}{b}, we apply H\"older's inequality noting \cref{def_ve_0} along with making use of \cref{2.9} and finally to obtain \redref{2.15.c}{c}, we make use of \cref{2.10}.

From Sobolev embedding  given in \cref{par_sob_emb} and properties of $\tilde{\zeta}_{i}$, we recall the estimate
\begin{equation}
\label{sob_dib}
\begin{array}{rcl}
\iint_{\tilde{Q}_{i}} \ukm{k_{i+1}}^{q} \tilde{\zeta}_i^{q} \,dz & \leq & \lbr \sup_{-\tilde{\tht}_i<t<\tilde{\tht}_i} \int_{{B}_{\tilde{\rho}_i}} \ukm{k_{i+1}}^2  \,dx \rbr^{\frac{p}{N}} \\
&& \times \lbr  \iint_{\tilde{Q}_i} |\nabla \ukm{k_{i+1}}|^p \,dz + \iint_{\tilde{Q}_i}  \ukm{k_{i+1}}^p |\nabla \tilde{\zeta}_i|^p \,dz\rbr.
\end{array}
\end{equation}

Our goal is to estimate each of the terms on the right hand side of \cref{sob_dib} using \cref{2.14} which we do as follows:
\begin{description}
\item[$\bullet$] Estimate for $\sup_{-\tilde{\tht}_i<t<\tilde{\tht}_i} \int_{{B}_{\tilde{\rho}_i}} \ukm{k_{i+1}}^2  \,dx$:  we make use of \cref{2.14} to get
\begin{equation}
\label{2.17}
\begin{array}{rcl}
\sup_{-\tilde{\tht}_i<t<\tilde{\tht}_i} \int_{{B}_{\tilde{\rho}_i}} \ukm{k_{i+1}}^2 \, dx & \leq & C \frac{2^{(i+2)(p+\ve_0)}}{(1-\sigma)^p }  \lbr   \frac{1}{\rho^pk^{\ve_0}} + \frac{1}{\tht k^{p+\ve_0-2}} \rbr \iint_{Q_i} \ukm{k_{i}}^{p+ \ve_0} \,dz.
\end{array}
\end{equation}

\item[$\bullet$] Estimate for $\iint_{\tilde{Q}_i} |\nabla \ukm{k_{i+1}}|^p \, dz$:  this term is also estimated from \cref{2.14} to get
\begin{equation}\label{2.18}
\begin{array}{rcl}
\iint_{\tilde{Q}_i} |\nabla \ukm{k_{i+1}}|^p \, dz & \leq & C \frac{2^{(i+2)(p+\ve_0)}}{(1-\sigma)^p }  \lbr   \frac{1}{\rho^pk^{\ve_0}} + \frac{1}{\tht k^{p+\ve_0-2}} \rbr \iint_{Q_i} \ukm{k_{i}}^{p+ \ve_0} \, dz.
\end{array}
\end{equation}
\item[$\bullet$] Estimate for $\iint_{\tilde{Q}_i}  \ukm{k_{i+1}}^p |\nabla \tilde{\zeta}_i|^p \ dz$: we estimate this term as follows:
\begin{equation}\label{2.19}
\begin{array}{rcl}
\iint_{\tilde{Q}_i}  \ukm{k_{i+1}}^p |\nabla \tilde{\zeta}_i|^p \,dz & \overset{\redlabel{2.19.a}{a}}{\leq} & \frac{2^{p(i+2)}}{(1-\sigma)^p \rho^p}\iint_{\tilde{Q}_i}  \ukm{k_{i+1}}^p \,dz\\
& \overset{\redlabel{2.19.b}{b}}{\leq} &\frac{2^{p(i+2)}}{(1-\sigma)^p \rho^p} \frac{2^{\ve_0(i+2)}}{k^{\ve_0}} \iint_{Q_i} \ukm{k_{i}}^{p+ \ve_0} \,dz.
\end{array}
\end{equation}
To obtain \redref{2.19.a}{a}, we made use of the bound $|\nabla \tilde{\zeta}_i| \leq \frac{2^{i+2} }{(1-\sigma)\rho}$ and to obtain \redref{2.19.b}{b}, we made use of \cref{2.11}.
\end{description}

Combining \cref{2.17}, \cref{2.18} and \cref{2.19} into \cref{sob_dib}, we get
\begin{equation}\label{2.20}
\begin{array}{rcl}
\iint_{\tilde{Q}_{i}} \ukm{k_{i+1}}^{q} \tilde{\zeta}_i^{q} \,dz & \leq &C \lbr[[] \frac{2^{(i+2)(p+\ve_0)}}{(1-\sigma)^p }  \lbr   \frac{1}{\rho^pk^{\ve_0}} + \frac{1}{\tht k^{p+\ve_0-2}} \rbr \iint_{Q_i} \ukm{k_{i}}^{p+ \ve_0} \,dz\rbr[]]^{\frac{p}{N}} \\
&& \times \frac{2^{(i+2)(p+\ve_0)}}{(1-\sigma)^p }  \lbr   \frac{1}{\rho^pk^{\ve_0}} + \frac{1}{\tht k^{p+\ve_0-2}} \rbr \iint_{Q_i} \ukm{k_{i}}^{p+ \ve_0} \,dz\\
& = & C \lbr[[] \frac{2^{(i+2)(p+\ve_0)}}{(1-\sigma)^p }  \lbr   \frac{1}{\rho^pk^{\ve_0}} + \frac{1}{\tht k^{p+\ve_0-2}} \rbr |Q_i|\, Y_i\rbr[]]^{1+\frac{p}{N}}.
\end{array}
\end{equation}
Note that $|Q_{i}| \approx |\tQ_i| \approx \rho^N \tht$. After dividing \cref{2.20} throughout by $|\tQ_i|$, we get
\begin{equation}
\label{2.21}
\begin{array}{rcl}
\fiint_{\tilde{Q}_{i}} \ukm{k_{i+1}}^{q} \tilde{\zeta}_i^{q} \,dz
& \leq & C \lbr[[] \frac{2^{(i+2)(p+\ve_0)}}{(1-\sigma)^p }  \lbr   \frac{1}{\rho^pk^{\ve_0}} + \frac{1}{\tht k^{p+\ve_0-2}} \rbr |Q_i|^{\frac{p}{N+p}}\, Y_i\rbr[]]^{1+\frac{p}{N}}\\
& \leq  & C \lbr[[] \frac{2^{(i+2)(p+\ve_0)}}{(1-\sigma)^p }  \lbr   \lbr \frac{\tht}{\rho^p}\rbr^{\frac{p}{N+p}}\frac{1}{k^{\ve_0}} + \lbr \frac{\rho^p}{\tht}\rbr^{\frac{N}{N+p}}\frac{1}{ k^{p+\ve_0-2}} \rbr Y_i\rbr[]]^{1+\frac{p}{N}}.
\end{array}
\end{equation}
Now substituting \cref{2.21} into \cref{2.15}, we get the following
\begin{equation}
\label{est_Y_i}
\begin{array}{rcl}
Y_{i+1} & \leq &C \lbr[[] \frac{2^{(i+2)(p+\ve_0)}}{(1-\sigma)^p }  \lbr   \lbr \frac{\tht}{\rho^p}\rbr^{\frac{p}{N+p}}\frac{1}{k^{\ve_0}} + \lbr \frac{\rho^p}{\tht}\rbr^{\frac{N}{N+p}}\frac{1}{ k^{p+\ve_0-2}} \rbr Y_i\rbr[]]^{\lbr 1+\frac{p}{N}\rbr \frac{p+\ve_0}{q}} \lbr[[] \frac{2^{(i+2)(p+\ve_0)}}{k^{p+\ve_0}}\, Y_i\rbr[]]^{\frac{q-(p+\ve_0)}{q}}\\
& = &C \frac{\lbr[[] {2^{(i+2)(p+\ve_0)}} \rbr[]]^{ 1+\frac{p}{N} \frac{p+\ve_0}{q}}}{(1-\sigma)^{p\lbr 1+\frac{p}{N}\rbr \frac{p+\ve_0}{q}}}\lbr   \lbr \frac{\tht}{\rho^p}\rbr^{\frac{p}{N+p}}\frac{1}{k^{\ve_0}} + \lbr \frac{\rho^p}{\tht}\rbr^{\frac{N}{N+p}}\frac{1}{ k^{p+\ve_0-2}} \rbr^{\lbr 1+\frac{p}{N}\rbr \frac{p+\ve_0}{q}} 
k^{(p+\ve_0)\frac{(p+\ve_0)-q}{q}}
Y_i^{ 1+\frac{p}{N} \frac{p+\ve_0}{q}}\\
& \overset{\redlabel{2.22.a}{a}}{\leq} &C \frac{\lbr[[] {2^{(i+2)(p+\ve_0)}} \rbr[]]^{ 1+\frac{p}{N} \frac{p+\ve_0}{q}}}{(1-\sigma)^{p\lbr 1+\frac{p}{N}\rbr \frac{p+\ve_0}{q}}}\lbr   \lbr \frac{\tht}{\rho^p}\rbr\frac{1}{k^{\ve_0\frac{N+p}{p}}} + \lbr \frac{\rho^p}{\tht}\rbr^{\frac{N}{p}}\frac{1}{ k^{(p+\ve_0-2)\frac{N+p}{p}}} \rbr^{ \frac{p}{N} \frac{p+\ve_0}{q}} 
k^{(p+\ve_0)\frac{(p+\ve_0)-q}{q}}Y_i^{ 1+\frac{p}{N} \frac{p+\ve_0}{q}}\\
& = & \mathbf{C}_0\frac{b^i }{\Upsilon(\sigma)}\aa_k^{ \frac{p}{N} \frac{p+\ve_0}{q}} k^{(p+\ve_0)\frac{(p+\ve_0)-q}{q}} Y_i^{ 1+\frac{p}{N} \frac{p+\ve_0}{q}} .
\end{array}
\end{equation}
To obtain \redref{2.22.a}{a}, we applied Jensen's inequality with exponent $\frac{N+p}{p}$. In the above estimate, 
$\mathbf{C}_0 = \mathbf{C}_0(N,p,\lamot)$ denotes a universal constant and we have set
\begin{equation}\label{a_k}
\begin{array}{rcl}
\aa_k &:= &  \lbr   \lbr \frac{\tht}{\rho^p}\rbr\frac{1}{k^{\ve_0\frac{N+p}{p}}} + \lbr \frac{\rho^p}{\tht}\rbr^{\frac{N}{p}}\frac{1}{ k^{(p+\ve_0-2)\frac{N+p}{p}}}\rbr,\\
b &:= & 2^{(p+\ve_0) \lbr 1 +  \frac{p}{N} \frac{p+\ve_0}{q}\rbr}, \\
\Upsilon(\sigma) & := & (1-\sigma)^{p\lbr 1+\frac{p}{N}\rbr \frac{p+\ve_0}{q}}.
\end{array}
\end{equation}

\section{Proof of \texorpdfstring{\cref{main_thm1}}.}
\label{F:imp_bound}
Let us make the choice 
\[
    \ve_0 = \frac{4}{N+2} \txt{and} \frac{2N}{N+2} < p < \infty,
\]
noting that for $q:= p\frac{N+2}{N}$, both the conditions in \cref{def_ve_0} are satisfied and hence all the estimates from \cref{sec3} are applicable.
Let us now set
\begin{equation}\label{def_al}
\al:= \frac{p}{N} \lbr\frac{p+\ve_0}{q}\rbr.
\end{equation}
If we now choose $k$ large enough such that $k \geq 1$, then $\aa_k$ as obtained in \cref{a_k} would be independent of $k$ since $\ve_0 \frac{N+p}{p}>0$ and $p+\ve_0 -2 >0$. In particular, we will have
\begin{equation}\label{def_aa_bnd}
\aa_k \leq \aa := \lbr \frac{\tht}{\rho^p} \rbr + \lbr \frac{\rho^p}{\tht} \rbr^{\frac{N}{p}}.
\end{equation}

\begin{remark}
    Indeed, if we wish to balance two terms on $\aa_k$ such that
\[
\lbr \frac{\tht}{\rho^p}\rbr\frac{1}{k^{\ve_0\frac{N+p}{p}}} \approx \lbr \frac{\rho^p}{\tht}\rbr^{\frac{N}{p}}\frac{1}{ k^{(p+\ve_0-2)\frac{N+p}{p}}}.
\]
This gives 
\[
 \rho^{p} k^{\ve_0} \approx \theta k^{p+\ve_0 -2},
\]
which is exactly what is obtained in \cite[Equation (12.2) of Section V]{DiBenedetto}. In particular,  they first determine $k$ to depend on $\rho$, $\theta$, and $p-2$ and later make $k$ large depending on other data which forces $p-2>0$. On the other hand, our approach removes this difficulty as long as $k \geq 1$.
\end{remark}

Using \cref{def_aa_bnd} into \cref{est_Y_i}, in order to make use of \cref{geo_con}, we see that  $Y_i \rightarrow 0$ as $i\rightarrow \infty$ provided
\begin{equation}\label{2.25}
Y_0 := \fiint_{Q_{\rho,\tht}}|u|^{p+\ve_0} \,dz \leq  \lbr[[] \frac{\aa^{ \frac{p}{N} \frac{p+\ve_0}{q}}}{(1-\sigma)^{p\lbr 1+\frac{p}{N}\rbr \frac{p+\ve_0}{q}}}\frac{\mathbf{C}_0}{k^{(p+\ve_0)\frac{q-(p+\ve_0)}{q}}}\rbr[]]^{-\frac{1}{\al}} b^{-\frac{1}{\al^2}},
\end{equation}
where $\al$ is from \cref{def_al}.  In particular, we can choose $k$ large enough such that equality holds in \cref{2.25}, i.e., the following equality holds:
\begin{equation}\label{2.26}
\begin{array}{rcl}
Y_0 := \fiint_{Q_{\rho,\tht}}|u|^{p+\ve_0} \,dz & =&  \lbr[[] \frac{\aa^{ \frac{p}{N} \frac{p+\ve_0}{q}}}{(1-\sigma)^{p\lbr 1+\frac{p}{N}\rbr \frac{p+\ve_0}{q}}}\frac{\mathbf{C}_0}{k^{(p+\ve_0)\frac{q-(p+\ve_0)}{q}}}\rbr[]]^{-\frac{1}{\al}} b^{-\frac{1}{\al^2}} \\
& = & \lbr[[] \frac{(1-\sigma)^{p\lbr 1+\frac{p}{N}\rbr \frac{p+\ve_0}{q}}}{\mathbf{C}_0\,\aa^{ \frac{p}{N} \frac{p+\ve_0}{q}}} \rbr[]]^{\frac{Nq}{p(p+\ve_0)}} k^{\frac{N(q-(p+\ve_0))}{p}} b^{-\lbr\frac{Nq}{p(p+\ve_0)}\rbr^2}.
\end{array}
\end{equation}
Henceforth, we shall fix the constant $k$ large enough such that \cref{2.26} holds which is possible since $p+\ve_0 < q$. Moreover, this also implies  $Y_0 < \infty$.

From the choice of $k$, we apply \cref{geo_con} to conclude
\[Y_{\infty} = \fiint_{Q_{\sigma \rho,\sigma \tht}} \ukm{k}^{p+\ve_0}\,dz = 0\] 
which is the same as 
\begin{equation*}
\sup_{Q_{\sigma \rho,\sigma \tht}} u \leq k \bigwedge 1.
\end{equation*} 
In particular, using \cref{2.26}, we have the following quantitative estimate
\begin{equation}\label{2.27}
\begin{array}{rcl}
\sup_{Q_{\sigma \rho,\sigma \tht}} u & \leq  & \mathbf{C} (N,p,\lamot) \frac{\aa^{\frac{p(N+2)}{2(p(N+2)-2N)}}}{(1-\sigma)^{\frac{p(N+p)(N+2)}{2\left(p(N+2)-2N\right)}}}  \lbr \fiint_{Q_{\rho,\tht}}|u|^{p+\ve_0} \,dz\rbr^{\frac{p(N+2)}{2(p(N+2)-2N)}} \bigwedge 1.
\end{array}
\end{equation}
Because $q > p+\ve_0 $, we can apply the local Sobolev embedding from \cref{par_sob_emb} to control the last term of \cref{2.27}. It is important to note that  the constant $\mathbf{C}$ in \cref{2.27} is stable in the range $p \in \lbr\frac{2N}{N+2}, \infty\rbr$.

\section{Proof of \texorpdfstring{\cref{main_thm2}}.}
\label{S:imp_bound}
In this section, we restrict our interest to the following region:
\begin{equation}
\label{3.1}
\ve_0 = \frac{2}{N+1} \txt{and} \frac{2N}{N+1} < p < \infty.
\end{equation}
With these choices, we see that all the results of \cref{sec3} and \cref{F:imp_bound} are applicable. In particular, we have
\begin{itemize}
    \item The bound $p+\ve_0 -2 > 0$ holds since
    \[p+ \ve_0 - 2 = p + \frac{2}{N+1} - 2 > \frac{2N}{N+1} + \frac{2}{N+1} - 2  = \frac{2N + 2}{N+1} - 2=0 .\]
    \item The bound $q > p+\ve_0$ holds since 
    \[2p- N\ve_0 =2p - \frac{2N}{N+1} > \frac{4N}{N+1} - \frac{2N}{N+1} =\frac{2N}{N+1} >0.\]
\end{itemize}

The proof of \cref{main_thm2} follows by iterating \cref{main_thm1} which we do as follows.

Let $\sigma \in (0,1)$ be given and fix the following cylinders:
\begin{equation}\label{3.2}
\rho_n:= \sigma \rho + (1-\sigma)\rho \sum_{i=1}^n \frac{1}{2^i} \txt{and}\tht_n:= \sigma \tht + (1-\sigma)\tht \sum_{i=1}^n \frac{1}{2^i},
\end{equation}
and the corresponding cylinders
\begin{equation*}
Q_n:= Q_{\sigma_n,\rho_n}, \txt{that gives} Q_0= Q_{\sigma \rho, \sigma \tht} \txt{and} Q_{\infty} = Q_{\rho,\tht}.
\end{equation*}
Let us set
\begin{equation*}
M_n:= \sup_{Q_n} u,
\end{equation*}


Let us now apply \cref{2.27} over the cylinders $Q_n$ and $Q_{n+1}$ to get
\begin{equation}\label{3.5}
\begin{array}{rcl}
M_n & \leq &  \mathbf{C}\, \aa^{\frac{p}{N (q-(p+\ve_0))}} \frac{2^{\frac{(n+1) p}{q-(p+\ve_0)}\lbr 1+\frac{p}{N} \rbr }}{(1-\sigma)^{\frac{p}{q-(p+\ve_0)}\lbr 1+\frac{p}{N} \rbr }}\lbr \fiint_{Q_{n+1}}|u|^{p+\ve_0} \,dz\rbr^{\frac{p}{N(q-(p+\ve_0))}} \bigwedge 1 \\
& \leq & \mathbf{C}\, \aa^{\frac{p}{N (q-(p+\ve_0))}} \frac{2^{\frac{(n+1) p}{q-(p+\ve_0)}\lbr 1+\frac{p}{N} \rbr }}{(1-\sigma)^{\frac{p}{q-(p+\ve_0)}\lbr 1+\frac{p}{N} \rbr }}\lbr\frac{|Q_{\infty}|}{|Q_{n+1}|} \fiint_{Q_{\infty}}|u|^{p} \,dz\rbr^{\frac{p}{N(q-(p+\ve_0))}} M_{n+1}^{\frac{\ve_0 p}{N(q-(p+\ve_0))}} \bigwedge 1.
\end{array}
\end{equation}
From \cref{3.2}, we see that
\begin{equation}\label{3.6}
\frac{|Q_{\infty}|}{|Q_{n+1}|} \leq \frac{|Q_{\infty}|}{|Q_{0}|} = C(N)\frac{\rho^N \tht}{\sigma \tht (\sigma \rho)^N}\leq C(N) \frac{1}{\sigma^{N+1}}.
\end{equation}
Combining \cref{3.5} and \cref{3.6}, we get
\begin{equation}\label{3.7}
M_n \leq \mathbf{C}\, \frac{\aa^{\frac{p}{N (q-(p+\ve_0))}}}{\sigma^{\frac{p(N+1)}{N(q-(p+\ve_0))}}} \frac{2^{\frac{(n+1) p}{q-(p+\ve_0)}\lbr 1+\frac{p}{N} \rbr }}{(1-\sigma)^{\frac{p}{q-(p+\ve_0)}\lbr 1+\frac{p}{N} \rbr }}\lbr \fiint_{Q_{\rho,\theta}}|u|^{p} \,dz\rbr^{\frac{p}{N(q-(p+\ve_0))}} M_{n+1}^{\frac{\ve_0 p}{N(q-(p+\ve_0))}} \bigwedge 1.
\end{equation}

Note that from the choice of $\ve_0 = \frac{2}{N+1}$ from \cref{3.1} and $q = p\frac{N+2}{N}$, we see that
\begin{equation*}
\frac{\ve_0 p}{N(q-(p+\ve_0))} = \frac{p}{p(N+1)-N} < 1 \txt{since} p > \frac{2N}{N+1} > 1.
\end{equation*}
Let us define the following terms:
\begin{gather*}
\mathbf{d}:=2^{{\frac{p(N+p)}{N(q-(p+\ve_0)) -\ve_0 p}} }, \\
\mathbb{B} := \frac{1}{\eta^{\frac{\ep_0 p}{N(q-(p+\ve_0))-\ve_0p}}}  \lbr \frac{\mathbf{C}  \,\aa^{\frac{p}{N (q-(p+\ve_0))}}}{\sigma^{\frac{p(N+1)}{N(q-(p+\ve_0))}}(1-\sigma)^{\frac{p}{q-(p+\ve_0)} \lbr 1+\frac{p}{N} \rbr }}\lbr \fiint_{Q_{\rho,\theta}}|u|^{p} \,dz\rbr^{\frac{p}{N(q-(p+\ve_0))}}  \rbr^{{\frac{N(q-(p+\ve_0))}{N(q-(p+\ve_0)) -\ve_0 p}}}.
\end{gather*}

Let us now fix an $\eta \in (0,1)$ such that 
\begin{equation}\label{3.11}
\eta \, \mathbf{d} = \frac{1}{2},
\end{equation} 
and apply  Young's inequality to \cref{3.7} with exponents 
\[ {\frac{N(q-(p+\ve_0))}{\ve_0 p}}  \txt{and}  {\frac{N(q-(p+\ve_0))}{N(q-(p+\ve_0)) -\ve_0 p}},\] 
 to obtain the following estimate
\begin{equation}\label{3.12}
M_n \leq \eta M_{n+1} + \mathbb{B} d^{n+1} \txt{for} n=0, 1,2,\ldots
\end{equation}
Iterating the estimate \cref{3.12} and noting the choice of $\eta$ in \cref{3.11}, we get
\begin{equation}\label{3.13}
M_0 \leq \eta^n M_{n+1} + \mathbb{B} \mathbf{d} \sum_{j=0}^{n} (\eta \mathbf{d})^j \xrightarrow{n \nearrow \infty} \mathbb{C}(n,p) \mathbb{B}.
\end{equation}
In particular, \cref{3.13} gives the following  quantitative bound
\begin{equation*}
\sup_{Q_{\sigma \rho, \sigma \tht}} u \leq  \mathbf{C} (N,p,\lamot)  \frac{ \aa^{\frac{p(N+1)}{2N(p-1)}}}{\sigma^{\frac{p(N+1)^2}{2N(p-1)}}(1-\sigma)^{\frac{p(N+p)(N+1)}{2N(p-1)} }}\lbr \fiint_{Q_{\rho,\tht}}|u|^{p} \,dz\rbr^{\frac{p(N+1)}{2N(p-1)}}   \bigwedge 1.
\end{equation*}

\begin{remark}
From the proof of \cref{main_thm2}, we see that  $\ve_0 = \frac{2}{N+1}$  is not the best choice of the exponent.  To ensure the above calculations work, we need to ensure the following three conditions are satisfied:
\begin{enumerate}[(i)]
\item\label{cond1} $p+\ve_0 -2 >0$
\item\label{cond2} $\frac{p(N+2)}{N} > p+ \ve_0 \Longleftrightarrow 2p > N \ve_0 \Longleftrightarrow 2p - N\ve_0 >0$.
\item\label{cond3} $\frac{\ve_0p}{N(q-(p+\ve_0))} < 1 \Longleftrightarrow \ve_0p < Nq - NP - N\ve_0$.
\end{enumerate}
All three conditions provides lower and upper bounds of $\ve_0$ such that:
\[
2-p < \ve_0 < \frac{2p}{N+p}
\]
and the lower bound matters only when $p<2$.  Thus for all the estimates from \cref{S:imp_bound} to hold, we would require the following bound to be satisfied by $\ve_0$ and $p$:
\[ g(\ve_0) = (2- \ve_0)^2 - N \ve_0 > 0 . \]
Then we observe that 
\[
g\left(\frac{2}{N+1}\right) = \frac{2N(N-1)}{(N+1)^2} \geq 0 \ \text{ and } \
g\left(\frac{4}{N+2}\right) = \frac{-8N}{(N+2)^2} < 0,
\]
which implies there exists a root $\de_0$ with $\frac{2N}{N+2} < 2-\delta_0 < \frac{2N}{N+1}$. Thus with $\ve_0 = \de_0$ and $2-\de_0 < p < \infty$, then all the calculations of \cref{S:imp_bound} carries over and analogous estimates can be recovered.

Since the explicit expression of $\de_0$ is not obtainable in a clean way, we made the choices of $\ve_0 = \frac{2}{N+1}$ and $\frac{2N}{N+1}<p<\infty$ for clarity of exposition.

\end{remark}

\section*{References}
%
%


\end{document}